\documentclass{amsart}
\usepackage{amsmath}
\usepackage{graphicx}
\usepackage{amsfonts}
\usepackage{amssymb}

\newtheorem{theorem}{Theorem}
\theoremstyle{plain}

\newtheorem{corollary}{Corollary}

\newtheorem{example}{Example}

\newtheorem{proposition}{Proposition}
\newtheorem{remark}{Remark}

\numberwithin{equation}{section}

\author{Samson Saneblidze}
\address{A. Razmadze Mathematical Institute\\
Department of Geometry and Topology\\
M. Aleksidze st., 1\\
0193 Tbilisi, Georgia}
 \email{sane@rmi.acnet.ge}

\dedicatory{To Nodar Berikashvili}

\thanks{This research described in this publication was made
possible in part by the grant \\
GNF/ST06/3-007 of the Georgian National Science Foundation}



\title
 {On the homotopy classification of maps}

\begin{document}

\begin{abstract}

We establish certain conditions which imply that a  map $f:X\rightarrow Y$ of
topological spaces is null homotopic when the induced integral cohomology homomorphism
is trivial; one of them is: $H^*(X)$ and $\pi_*(Y)$ have no torsion and $H^*(Y)$ is
polynomial.

\end{abstract}
\date{}
\maketitle

\section{Introduction}
We give certain classification theorems for maps via induced cohomology homomorphism.
Such a classification is based on a new aspects of  obstruction theory to the
section problem in a fibration  beginning  in \cite{berikaBUL},\,\cite{berikaOBST} and
 developed in some directions
  in \cite{saneMono},\,\cite{saneWeak}. Given a fibration  $F\rightarrow E\overset{\xi}{\rightarrow }X,$
  the  obstructions to the section problem of $\xi$ naturally lay in the groups  $H^{i+1}(X;\pi_i(F)),i\geq 0.$
    A basic method here is to use the Hurewicz homomorphism $u_i:
\pi_i(F)\rightarrow H_i(F)$ for passing the above obstructions into the groups  $H^{i+1}(X;H_i(F)),i\geq 0.$
     In particular, this suggests the following condition on a fibration:
 The induced homomorphism
\begin{equation*}\label{HH}
\hspace{-0.8in}(1.1)_m \ \ \ \ \ \ \ \ \  \ \ \ \       u^*:H^{i+1}(X;\pi_i(F))\rightarrow H^{i+1}(X;H_i(F)),\  1\leq i<m,
\end{equation*}
is an inclusion (assuming $u_1:\pi_1(F)\rightarrow H_1(F)$ is an isomorphism). Note
also that the idea of using the Hurewicz map in the obstruction theory goes back to the
paper \cite{Pontrjagin}. (Though its main result  was erroneous, it became  one crucial
point for applications of  characteristic classes (see \cite{D-W}).)

 For the
homotopy classification of maps $X\rightarrow Y,$  the space $F$ in  $(1.1)_m$ is replaced by $\Omega
Y$  and  we establish the following statements.
 Below all topological spaces are assumed to be path connected (hence, $Y$ is also simply connected) and the ground coefficient
ring is the integers ${\mathbb Z}.$ Given a  commutative graded algebra (cga)
$H^*$ and an integer $m\geq 1,$ we say that $H^*$ is \emph{$m$-relation free} if $H^i$ is torsion free  for
 $ i\leq m$  and also there
is no multiplicative relation in $H^i$ for  $i\leq m+1;$
in particular,   $H^{2i-1}=0$ for $1\leq i\leq [\frac{m+2}{2}].$
 We also allow  $m=\infty$ for
$H$ to be polynomial on even degree generators.
\begin{theorem}\label{null}
 Let $f:X\rightarrow Y$ be a map such that the pair $(X,\Omega Y)$ satisfies $(1.1)_m,$
$X$ is an $m$-dimensional polyhedron and
   $H^*(Y)$ is  $m$-relation free.
Then  $f$ is null homotopic if and only if
 \[0=H^*(f): H^*(Y)\rightarrow H^*(X).\]
\end{theorem}
\begin{theorem}\label{null2}
Let $X$ and $Y$ be spaces such that the Hurewicz map $u_i: \pi_i(\Omega Y)\rightarrow
H_i(\Omega Y)$ is an inclusion for $1\leq i<m,$ and
$\operatorname{Tor}\left(H^{i+1}(X), H_i(\Omega Y)/\pi_i(\Omega Y)\right)=0$ when
$\pi_{i}(\Omega Y)\neq 0,$ $X$ is an $m$-dimensional polyhedron and
   $H^*(Y)$ is  $m$-relation free.
 Then a map $f:X\rightarrow Y$ is null
homotopic if and only if
 \[0=H^*(f): H^*(Y)\rightarrow H^*(X).\]
\end{theorem}

\begin{theorem}\label{null3}
Let $X$ be an $m$-dimensional polyhedron   and $G$
 a topological group such that  $\pi_i(G)$ is torsion free for $1\leq i<m,$
and $\operatorname{Tor}\left(H^{i+1}(X), \operatorname{Coker}u_i\right)=0,\,$
  $u_i:\pi_i(G)\rightarrow H_i(G)$ when $\pi_i(G)\neq 0.$ Suppose that   the cohomology algebra $H^*(BG)$ of the classifying
space $BG$ is $m$-relation free. Then a map $f:X\rightarrow BG$ is null homotopic if and only
if
 \[0=H^*(f): H^*(BG)\rightarrow H^*(X).\]
\end{theorem}
In fact the two last Theorems follow from the first one, since  their hypotheses imply
$(1.1)_m,$ too. A main example of $G$ in Theorem 3 is the unitary group $U(n)$ with
$m=2n,$   since $u_{2i}$ is a trivial inclusion  and  $u_{2i-1}$ is an inclusion given
by multiplication by the integer $(i-1)!$ for  $1\leq i\leq n.$ A $U(n)$-principal
fibre bundle over $X$ is classified by a map $X\rightarrow BU(n).$ Suppose that all its
Chern classes are trivial, then $H^*(f)=0$ and by Theorem 3, $f$ is null homotopic.
Therefore the $U(n)$-principal fibre bundle is trivial. Thus, we  have in fact deduced
the following statement, the main result of \cite{Peterson} (compare also
\cite{Ethomas}).
\begin{corollary}
Let $\xi$ be a $U(n)$-principal fibre bundle over $X$ with $\dim X\leq 2n$ and the only torsion in $H^{2i}(X)$ is relatively prime to $(i-1)!.$
Then $\xi$ is trivial if and only if the Chern classes $c_k(\xi)=0$ for $1\leq k \leq n.$
\end{corollary}
While the proof of this statement  in \cite{Peterson} does not admit an
immediate generalization for an infinite dimensional $X,$ Theorem 3 does by taking
$m=\infty.$ Furthermore, for  $G=U$ and  $X=BU$ recall that $[BU,BU]$ is an abelian
group, so  we get that two maps $f,g:BU\rightarrow BU$ are homotopic if and only if
$H^*(f)=H^*(g): H^*(BU;\mathbb Q)\rightarrow H^*(BU;\mathbb Q)$ (compare \cite{J-M-O},
\cite{notbohm}). Note also that when $m=\infty$ in Theorem 3,  $H^*(Y)$ must have
infinitely many polynomial generators (e.g. $Y=BU,BSp$) as it follows from the solution
of
   the Steenrod
problem  for finitely generated polynomial rings \cite{Ande-Grod}
(the underlying spaces do not have  torsion free homotopy groups in all degrees).

Finally, note that beside  obstruction theory we apply a main ingredient of the proof of Theorem
1 is an explicit form of  minimal multiplicative (non-commutative) resolution of an
$m$-relation free cga (of a polynomial  algebra when $m=\infty$) in total degrees
$\leq m$ (compare \cite{saneMono}, \cite{saneFilt}). Namely, the generator set of the
resolution in the above range  only consists of monomials formed by
 $\smile_1$ products. Remark  that the idea of using  $\smile_1$ product when dealing
 with polynomial cohomology,
  especially  in the context  of homogeneous spaces,
 has been realized  by several authors
\cite{May}, \cite{Gug-May},  \cite{Munkholm}, \cite{H-M-S}
 (see also \cite{McCleary2} for further references).

In sections 2 and 3 we recall certain basic definitions and constructions, including
the functor $D(X;H_*)$ \cite{berika},\,\cite{berikaZUR}, for the aforementioned
obstruction theory, and in section 4 prove  Theorems 1-3.

I am grateful to Jesper Grodal for helpful comments. I  thank to Jim Stasheff for
helpful comments and suggestions.

\newpage

\section{Functor D(X;H)}

Given  a bigraded differential algebra  $A=\{A^{i,j}\}$  with $d:A^{i,j}\rightarrow
A^{i+1,j}$ and total degree $n=i+j,$ let
 $D(A)$ be  the set \cite{berikaZUR}
   defined by $D(A)=M(A)/G(A)$ where
\[
\begin{array}{llll}
M(A)&=&\{a\in A^{1}\,|\,da=-aa, \,\, a=a^{2,-1}+a^{3,-2}+\cdots \},\\
G(A)&=&\{p\in A^{0}\,|\, p=1+ p^{1,-1}+p^{2,-2}+\cdots   \},\\
\end{array}
\]
and the action $M(A)\times G(A)\rightarrow M(A)$ is given by the  formula
\begin{equation}\label{gauge}
a\ast p=p^{-1}ap+ p^{-1}dp.
\end{equation}
 In other words, two elements $a,b\in M(A)$ are on the
same orbit if there is $p\in G(A),\,  p=1+p',$ with
\begin{equation}\label{action}
b-a=ap' -p'b+dp'.
\end{equation}
Note that an element $a=\{a^{*,*}\}$ from $M(A)$ is of total degree $1$ and  referred
to as \emph{twisting}; we usually  suppress the second degree below. There is   a
distinguished element in the set $D(A),$ the class of $0\in A,$ and denoted by the same symbol.

There is simple but useful
  (cf. \cite{saneMono})
\begin{proposition}\label{derivation}
 Let $f,g :A^{*,*}\rightarrow B^{*,*}$  be two
  dga maps  that preserve the bigrading. If they
    are $(f,g)$-derivation homotopic via $s:A^{i,j}\rightarrow B^{i-1,j},$ i.e.,
$f-g=sd+ds$ and $s(ab)=(-1)^{|a|}fasb+sagb,$
     then
  $D(f)=D(g):D(A)\rightarrow D(B).$
\end{proposition}
\begin{proof}
Given $a\in M(A),$   apply   the $(f,g)$-derivation homotopy $s$ to get
$fa-ga=dsa+sda=dsa+s(-aa)=dsa+fasa-saga.$ From this we deduce that $fa$ and $ga$ are
equivalent by (\ref{action}) for  $p'=-sa.$
\end{proof}

Another useful property of $D$ is fixed by the following   comparison theorem
\cite{berika},\,\cite{berikaZUR}:
\begin{theorem}\label{comparison}
 If  $f:A\rightarrow B$ is   a cohomology isomorphism, then   $D(f):D(A)\rightarrow D(B)$
is a bijection.
\end{theorem}

For our purposes the main example  of $D(A)$ is the following (cf.
\cite{berika},\,\cite{berikaZUR})
\begin{example}\label{DX}
Fix a graded (abelian) group $H_{*}.$ Let
\[\rho : (R_{\geq 0}H_{q}, {\partial}^{R})\rightarrow
H_{q},\,\,\,{\partial}^{R }:R_{i}H_{q}\rightarrow R_{i-1}H_{q},\] be its free
group resolution. Form the  bigraded $\operatorname{Hom}$ complex
 $$(\mathcal{R}^{*,*},d^{R})=
\left(Hom(RH_{*}\,,RH_{*}),d^{R}\right),\ \ \ \ d^{R}: \mathcal{R}^{s,t}\rightarrow
\mathcal{R}^{s+1,t};$$ an element $f\in \mathcal{R}^{*,*}$ has bidegree $(s,t)$ if
  $f: R_{j}H_{q}\rightarrow
R_{j-s}H_{q-t}.$  Note also that $\mathcal{R}^{*,*}$ becomes  a dga with
respect to the composition product.

 Given a  topological space $X,$ consider the dga
\[(\mathcal H,\nabla)=( C^{*}(X;\mathcal{R}),\nabla
=d^{C}+d^{R})\]
 which is bigraded via
 $\mathcal H^{r,t}=\prod
_{r=i+j}C^{i}(X;\mathcal{R}^{j,t}).$ Thus we get
\[\mathcal H=\{\mathcal H^{n}\},\ \ \ \ \  \mathcal
H^{n}=\prod_{n=r+t}\mathcal H^{r,t},\ \ \  \ \nabla : \mathcal H^{r,t} \rightarrow
\mathcal H^{r+1,t}.\]
 We refer to
 $r$ as the {\em perturbation} degree which is
mainly exploited by inductive  arguments below. For example, for a twisting cochain
$h\in M(\mathcal{H}),$ we have \[ h=h^{2} + \dotsb +h^{r}+\dotsb
,\,\,\,\,\,\,\,\,\,\,h^{r}\in \mathcal H^{r,1-r},\] satifying  the following sequence
of equalities:
\begin{equation}\label{sequence}
\nabla (h^{2})=0, \,\,\,\,\nabla (h^{3})=-h^{2}h^{2},\,\,\,\,\nabla
(h^{4})=-h^{2}h^{3}-h^{3}h^{2},\dotsc .
\end{equation}

Define $$D(X;H_{*})=D(\mathcal H,\nabla).$$
 Then $D(X;H_{*})$
   becomes a functor on the category of topological spaces and
 continuous  maps  to the category  of  pointed sets.
\end{example}

\begin{example}\label{example2}
Given two dga's $B^*$ and $C^{*,*}$ with $d^B: B^i\rightarrow B^{i+1}$ and $d^C_1:
C^{j,t}\rightarrow C^{j+1,t},\,$ $d^C_2=0,$ let $A=B\hat{\otimes}C.$ View $(A,d)$ as
bigraded via $A=\{A^{r,t},d\},$ $A^{r,t}=\prod_{r=i+j}B^i\otimes C^{j,t},$
$d=d^B\otimes 1+1\otimes d^C_1.$
 Note also that the dga $(\mathcal{H},\nabla)$ in the previous example can also be
 viewed  as a special case of the above tensor product algebra by setting
 $B^*=C^*(X)$ and $C^{*,*}=\mathcal{R}^{*,*}.$
\end{example}

\section{Predifferential $d(\xi)$ of a fibration}

Let $F\rightarrow E\overset{\xi}{\longrightarrow} X$ be a fibration.  In \cite{berika} a unique element of
$D(X;H_*(F))$  is naturally  assigned to $\xi;$
 this element is denoted by $d(\xi)$ and referred to
as the \emph{predifferential} of $\xi.$ The naturalness  of $d(\xi)$ means that for a
map $f:Y\rightarrow X,$
\begin{equation}\label{pred}
d(f(\xi))= D(f)(d(\xi)),
\end{equation}
where $f(\xi)$ denotes the induced fibration on $Y.$

Originally $d(\xi)$ appeared in homological perturbation theory for measuring
the non-freeness of the Brown-Hirsch model: First, in \cite{hirsch}
G. Hirsch modified E. Brown's twisting tensor
product model $(C_*(X)\otimes C_*(F),d_{\phi})\rightarrow (C_*(E),d_E)$
\cite{Brown}, \cite{Gugenheim} by replacing the chains $C_*(F)$
 by its homology $H_*(F)$ provided the homology is a free module.
 In \cite{berika} the Hirsch model was extended for arbitrary $H_*(F)$
  by replacing it by a free module resolution $RH_*(F)$
   to obtain $(C_*(X)\otimes RH_*(F),d_h)$ in which
   $d_h=d_X\otimes 1+1\otimes d_{F}+-\cap h$ and
    $h$ is just an element of $M(\mathcal{H})$ in
     Example \ref{DX} with $H_*=H_*(F).$ Furthermore,
     to an  isomorphism
      $p:(C_*(X)\otimes RH_*(F),d_h)\rightarrow (C_*(X)\otimes RH_*(F),d_{h'})$
      between two such models answers an equivalence relation
$h\sim_{p}h'$ in $M(\mathcal{H}),$ and the class of $h$ in $D(X;H_*(F))$ is identified
as $d(\xi).$ More precisely, we recall some basic constructions for the definition of
$d(\xi)$ we need for the obstruction  theory in question.

For convenience, assume that $X$ is a polyhedron and that $\pi_1(X)$ acts trivially on
$H_*(F).$ Then $\xi$ defines the following colocal system of   chain complexes over
$X:$ To each simplex $\sigma \in X$ is assigned the singular chain complex $(C_{*}(F_{\sigma}),\gamma _{\sigma})$ of the space
$F_{\sigma}=\xi^{-1}(\sigma):$
\[    X \ni     \sigma   \longrightarrow (C_{*}(F_{\sigma}),\gamma _{\sigma})\subset (C_*(E),d_E),\] and  to  a
pair $\tau\subset \sigma$ of simplices  an   induced chain map
\[C_{*}(F_{\tau})\rightarrow C_{*}(F_{\sigma}).\] Set
$\mathcal{C}_{\sigma}=\{\mathcal{C}^{s,t}_{\sigma}\},\, \mathcal{C}^{s,t}_{\sigma}
=\operatorname{Hom}^{s,t}(R_*H_{*}(F),C_{*}(F_{\sigma}))$ where $C_*$ is regarded as
bigraded via $C_{0,*}=C_*, C_{i,*}=0,\,i\neq 0,$ and $f:R_jH_q(F)\rightarrow
C_{j-s,q-t}(F_{\sigma})$ is of bidegree $(s,t).$ Then we obtain  a  colocal system of
cochain complexes $\mathcal{C}=\{\mathcal{C}^{*,*}_{\sigma}\}$ on $X.$
  Define $\mathcal{F}$   as the simplicial
cochain complex $C^*(X;\mathcal{C})$ of $X$ with coefficients  in  the  colocal system
$\mathcal{C}.$ Then
\[
\mathcal F=\{{\mathcal F}^{i,j,t}\},\,\,\,\,{\mathcal
F}^{i,j,t}=C^{i}(X;\mathcal{C}^{j,t}). \] Furthermore, obtain the bicomplex
$\mathcal{F}=\{\mathcal F^{r,t}\} $  via \[\mathcal F^{r,t}=\prod _{r=i+j}\mathcal
F^{i,j,t},\ \delta:{\mathcal F}^{r,t}\rightarrow {\mathcal F}^{r+1,t},\ \gamma
:{\mathcal F}^{r,t}\rightarrow {\mathcal F}^{r,t+1},\
 \delta
=d^{C}\!\!+\partial^{R},\ \gamma =\{\gamma_{\sigma}\},\] and finally set
\[\mathcal
{F}=\{\mathcal {F}^{m}\}, \ \ \mathcal {F}^{m}=\prod_{m=r+t}\mathcal {F}^{r,t}.\]
   We have a natural dg pairing
$$(\mathcal{F},\delta+\gamma)\otimes     ({\mathcal   H},\nabla)\rightarrow
({\mathcal F},\delta+\gamma)$$
 defined by
 $\smile$ product on $C^*(X;-)$ and  the obvious pairing
$\mathcal{C}_{\sigma}\otimes \mathcal{R}\rightarrow \mathcal{C}_{\sigma} $
 in coefficients; in particular we have  $\gamma (fh) = \gamma (f)h$
 for $f\otimes h\in \mathcal{F}\otimes \mathcal{H}.$
 Denote $\mathcal{R}_{_{\#}}=Hom(RH_{*}(F),H_{*}(F))$ and define
\[({\mathcal F}_{_{\#}},\delta_{_{\#}}):=
 (H (\mathcal F,\gamma),\delta_{_{\#}})=(C^{*}(X;\mathcal{R}_{_{\#}}) ,\delta_{_{\#}}).
 \]
Clearly,  the above pairing  induces the following dg pairing \[({\mathcal
F}_{_{\#}},\delta_{_{\#}}) \otimes (\mathcal H,\nabla)\rightarrow  ({\mathcal
F}_{_{\#}},\delta_{_{\#}}).\]
 In
other words, this pairing is also defined by  $\smile$ product on $C^*(X;-)$ and
the
 pairing $\mathcal{R}_{_{\#}}\otimes \mathcal{R}\rightarrow \mathcal{R}_{_{\#}} $
 in coefficients.
Note that $\rho$ induces an epimorphism of chain complexes
 \[\rho^{*}: (\mathcal
H,\nabla)\rightarrow (\mathcal F_{_{\#}},\delta_{_{\#}}).\]
 In turn, $\rho^*$ induces an isomorphism in cohomology.

Consider the following equation
 \begin{equation}\label{equation} (\delta+\gamma)(f) = fh
\end{equation} with respect to a pair $(h,f)\in \mathcal{H}^1\times \mathcal{F}^{0} ,$
$$
\begin{array}{lll}
h= h^{2}+\dotsb+ h^{r}+\dotsb ,    & h^{r}\in {\mathcal H}^{r,1-r}, \\
f= f^{0}+ \dotsb + f^{r}+ \dotsb ,  & f^{r} \in {\mathcal F}^{r,-r},
\end{array}
$$
satisfying
 the  initial conditions:
$$
\begin{array}{llll}
\nabla(h)=-hh &\\

 \gamma (f^{0})=0, & [f^{0}]_{\gamma}=\rho ^{*}(1)\in {\mathcal F}^{0,0}_{_{\#}},\ \ 1\in \mathcal H.
\end{array}
$$
Let $(h,f)$ be a solution of the above equation.
 Then   $d(\xi)\in D(X;H_*(F))$ is defined  as the class of $h.$
   Moreover, the transformation of $h$ by (\ref{gauge}) is extended to
pairs $ (h,f)$ by the map
\[   (M(\mathcal{H})\times \mathcal{F}^{0})\times (G(\mathcal{H})\times
\mathcal{F}^{-1})\rightarrow M(\mathcal{H})\times \mathcal{F}^{0} \] given for $((h,f),
(p,s))\in(M(\mathcal{H})\times \mathcal{F}^{0})\times (G(\mathcal{H})\times
\mathcal{F}^{-1})$
  by the
formula
\begin{equation}\label{pair}
(h,f)\ast (p,s)=\left(h\ast p\,,\,   fp+ s(h\ast p)+(\delta+\gamma)(s)\right).
\end{equation}
We have that a solution $(h,f)$ of the equation exists and is unique up to the above
action. Therefore, $d(\xi)$ is well defined.

Note that action (\ref{pair}) in particular has a property that if $(\bar{h},\bar
{f})=(h,f)\ast (p,s)$ and $h^r=0 $ for $2\leq r\leq n,$ then in view of (\ref{action})
one gets the equalities
\begin{equation}\label{zero}
\bar{h}^{n+1}=h\ast (1+ p^n)=h^{n+1}+\nabla(p^n).
\end{equation}

\subsection{Fibrations  with $d(\xi)=0$  }
 The main fact of this subsection is the following theorem from \cite{berikaBUL}:
\begin{theorem}\label{section}
Let $F\rightarrow E\overset{\xi}{\longrightarrow} X$ be a fibration such that $(X,F)$
satisfies $(1.1)_m$. If the restriction of $d(\xi)\in D(X;H_*(F))$ to $d(\xi)|_{X^m}\in
D(X^m;H_*(F))$ is zero, then $\xi$ has a section on the $m$-skeleton of $X.$ The case
of $m=\infty,$ i.e., $d(\xi)=0,$ implies the existence of a  section on $X.$
\end{theorem}
\begin{proof}Given a pair $(h,f)\in \mathcal{H}\times \mathcal{F},$
let $(h_{tr},f_{tr})$ denote  its  component    that lies in
\begin{equation*}
 C^{*}(X;Hom(H_0(F),RH_*(F)))\times
  C^{*}(X;Hom(H_0(F),C_*(F_{\sigma}))).
\end{equation*}
Below $(h_{tr},f_{tr})$ is referred to as the \emph{transgressive} component of
$(h,f).$ Observe that since $RH_0(F)=H_0(F)=\mathbb Z,$ we can  view
$(h^{r+1}_{tr},f^r_{tr})$ as a pair of cochains laying in $ C^{>r}(X;RH_r(F))\times
C^{r}(X;C_r(F_{\sigma})). $ Such an interpretation allows us to identify a section
$\chi^r:X^r\rightarrow E$ on the $r$-skeleton $X^r\subset X$   with a cochain, denoted
by $c^r_{\chi},$ in $C^{r}(X;C_r(F_{\sigma}))$ via
$c^r_{\chi}(\sigma)=\chi^r|_{\sigma}:\Delta^r \rightarrow F_{\sigma}\subset E,$
 $\sigma\subset X^r$ is an $r$-simplex, $r\geq 0.$

The proof of the theorem just consists of choosing a solution $(h,f)$ of
(\ref{equation}) so that the transgressive component $f_{tr}=\{f_{tr}^r\}_{r\geq 0}$ is
specified by $f^r_{tr}=c^r_{\chi}$  with $\chi$ a section of $\xi.$  Indeed, since $F$
is path connected, there is a section $\chi^1$ on $X^1;$ consequently, we get the pairs
$(0,f^0_{tr}):=(0,c^0_{\chi})$ and $(0,f^1_{tr}):=(0,c^1_{\chi})$ with
$\gamma(f^1_{tr})=\delta(f^0_{tr}).$ Then $\delta(f^1_{tr})\in C^2(X;C_1(F))$ is a
$\gamma$-cocycle and $[\delta(f^1_{tr})]_{\gamma}\in C^2(X;H_1(F))$ becomes the
obstruction cocycle $c(\chi^1)\in C^2(X;\pi_1(F))$ for extending  of $\chi^1$ on $X^2;$
moreover, one can choose $h^2_{tr}$ to be satisfying
$\rho^*(h^2_{tr})=[\delta(f^1_{tr})]_{\gamma}$ (since $\rho^*$ is an epimorphism and a
weak equivalence).

Suppose by induction that we have constructed a solution $(h,f)$ of (\ref{equation})
and  a section $\chi^n$ on $X^n$ such that $h^r=0$  for $ 2\leq r\leq n,$ $f^n_{tr}=c^n_{\chi}$
 and \[\rho^*(h^{n+1}_{tr})=[\delta(f^n_{tr})]_{\gamma}\in  C^{n+1}(X;H_n(F)).\] In view of
(\ref{sequence}) we have $\nabla(h^{n+1})=0$ and from the above equality immediately
follows that \[u^{\#}(c(\chi^n))=\rho^*(h^{n+1}_{tr})\]
in which
$c(\chi^n)\in C^{n+1}(X;\pi_n(F))$ is  the obstruction  cocycle for extending  of $\chi^n$ on $X^{n+1}$  and
$u^{\#}:C^{n+1}(X;\pi_n(F))\rightarrow C^{n+1}(X; H_n(F)).$

Since $d(\xi)|_{X^m}=0,$ there is $p\in G(\mathcal{H})$ such that $(h\ast p)|_{X^m}=0;$
in particular, $(h\ast p)^{n+1}=0\in \mathcal{H}^{n+1,-n}$ and in view of (\ref{zero})
we establish the equality $h^{n+1}=-\nabla(p^n),$ i.e., $[h^{n+1}]=0\in
H^*(\mathcal{H},\nabla);$ in particular, $[h^{n+1}_{tr}]=0\in H^{n+1}(X;H_n(F)).$
Consequently, $[u^{\#} (c(\chi^n))]=0\in H^{n+1}(X; H_n(F)).$ Since $(1.1)_n$ is an
inclusion induced by $u^{\#},$ $[c(\chi^n)]=0\in H^{n+1}(X;\pi_n(F)).$ Therefore, we
can extend $\chi^n$ on $X^{n+1}$ without changing it on $X^{n-1}$ in a standard way.
Finally, put $f^{n+1}_{tr}=c^{n+1}_{\chi}$ and choose a $\nabla$-cocycle $h^{n+2}_{tr}$
satisfying $\rho^*(h^{n+2}_{tr})=[\delta(f^{n+1}_{tr})]_{\gamma}.$ The induction step
is completed.
\end{proof}

\section{Proof of Theorems \ref{null}, \ref{null2} and \ref{null3}}
First we recall  the following application of Theorem \ref{section} (\cite{berikaBUL})

\begin{theorem}\label{Dnull}
Let $f:X\rightarrow Y$ be a map such that $X$ is an $m$-polyhedron and
 the pair $(X,\Omega Y)$ satisfies $(1.1)_m.$
If\, $0=D(f):D(Y;H_*(\Omega Y))\rightarrow D(X;H_*(\Omega Y)),$ then $f$ is null
homotopic.
\end{theorem}

\begin{proof} Let $\Omega \rightarrow PY\overset{\pi}{\rightarrow} Y$ be the path
fibration and $f(\pi)$  the induced fibration. It suffices to show that $f(\pi)$ has a
section. Indeed,  (\ref{pred}) together with  $D(f)=0$ implies $d(f(\pi))=0,$ so
Theorem \ref{section} guaranties the existence of  the section.
\end{proof}

Now we are ready to prove the theorems stated in the introduction. Note that just below
we shall heavily use  multiplicative, non-commutative resolutions of  cga's that are
enriched with $\smile_1$ products. Namely,
 given a space $Z,$  recall  its  filtered model $f_Z:(RH(Z),d_h)\rightarrow C^*(Z)$
\cite{saneMono},\,\cite{saneFilt} in which  the underlying differential (bi)graded
algebra $(RH(Z),d)$ is a non-commutative version of
 Tate-Jozefiak resolution of the cohomology algebra $H^*(Z)$
(\cite{Tate},\,\cite{Jozefiak}), while $h$ denotes a perturbation of $d$  similar to
\cite{hal-sta}. Moreover, given  a map $X\rightarrow Y,$ there is  a dga map $
RH(f):(RH(Y),d_h) \rightarrow (RH(X),d_h)$ (not uniquely defined!) such that the
following diagram
\begin{equation}\label{diagram}
\begin{array}{ccccc}
  (RH(Y),d_h) & \overset{RH(f)}\longrightarrow & (RH(X),d_h) \vspace{1mm}\\
       ^{_{f_Y}} \!\! \downarrow  & &     \downarrow  ^{_{f_X}}  \\
C^*(Y) & \overset{C(f)}\longrightarrow & C^*(X)
\end{array}
\end{equation}
 commutes up to $(\alpha,\beta)$-derivation homotopy
with $\alpha=C(f)\circ f_Y$ and $\beta=f_X\circ RH(f)$
 (see, \cite{hueb},\,\cite{saneMono}).

\vspace{0.1in}

 \noindent\emph{Proof of Theorem
\ref{null}.} The non-trivial part of the proof is to show that $H(f)=0$ implies  $f$ is
null homotopic. In view of Theorem \ref{Dnull} it suffices to show that $D(f)=0.$
 By (\ref{diagram}) and  Proposition \ref{derivation} we
get the commutative diagram of pointed sets
$$
\begin{array}{ccccc}
  D(\mathcal{H}_Y) & \overset{D(\mathcal{H}(f))}\longrightarrow & D(\mathcal{H}_{X})
   \vspace{1mm}\\
      ^{ _{D(f_Y)}}  \downarrow  & &     \downarrow ^{_{D(f_X)}}  \\
D(Y;H_*(\Omega Y)) & \overset{D(f)}\longrightarrow & D(X;H_*(\Omega Y))
\end{array}
$$
in which $$\ \mathcal{H}_{X}=RH^*(X)\hat{\otimes} Hom(RH_*(\Omega Y)\,,RH_*(\Omega
Y)),$$
$$\mathcal{H}_{Y}=RH^*(Y)\hat{\otimes} Hom(RH_*(\Omega Y)\,, RH_*(\Omega Y))  $$
 (see Example \ref{example2}) and  the vertical maps are  induced by $f_X\otimes 1$
 and $f_Y\otimes 1;$ these maps are
 bijections by Theorem \ref{comparison}. Below we need an explicit form of $RH(f)$ to see
that $H(f)=0$ necessarily  implies $RH(f)|_{V^{(m)}}=0$ with $V^{(m)}=\bigoplus_{1\leq
i+j \leq m} V^{i,j};$ hence, the restriction of the map $\mathcal{H}(f):=RH(f)\otimes
1$ to $RH^{(m)}\otimes 1,$ $RH^{(m)}=\bigoplus_{1\leq i+j\leq m}R^iH^j(Y),$ is zero,
and, consequently,
\begin{equation}\label{composition}
 D(f_X)\circ D(\mathcal{H}(f))=0.
\end{equation}

First observe that any multiplicative resolution $(RH,d)=(T(V^{*,*}),d),\,
V=\langle\mathcal{V}\rangle,$ of a cga $H$ admits a sequence of multiplicative
generators, denoted by
\begin{equation}\label{cups1}
a_1\smile_1\cdots \smile_1a_{n+1}\in \mathcal{V}^{-n,*},\ \  a_i\in \mathcal{V}^{0,*},   \ \ n\geq 1,
\end{equation}
 where   $a_i
\smile_1a_j=(-1)^{(|a_i|+1)(|a_j|+1)}a_j\smile_1 a_i $ and $a_i\neq a_j$ for $i\neq j.$
Furthermore, the expression $ab\smile _1 uv$ also has a sense
    by means of formally (successively) applying the Hirsch formula
\begin{equation}\label{hirsch}
c\smile _{1}(ab)=(c\smile _{1}a)b+(-1)^{|a|(|c|+1)}a(c\smile _{1}b).
\end{equation}
 The resolution differential $d$  acts on  (\ref{cups1})  by  iterative application
  of the formula
\[d(a\smile_1b)=da\smile_1 b-(-1)^{|a|}a\smile_1db +(-1)^{|a|}ab-(-1)^{|a|(|b|+1)}ba .\]
 Consequently, we get
\[d(a_1\smile_1\cdots \smile_1a_{n})= \sum_{(\mathbf{i};\mathbf{j})} (-1)^{\epsilon}
 (a_{i_1}\smile_1\cdots \smile_1a_{i_k})\cdot(a_{j_1}\smile_1\cdots \smile_1a_{j_\ell})\]
where the summation is over unshuffles $(\mathbf{i};\mathbf{j})=(i_1<\cdots <i_k\,
;j_1<\cdots<j_{\ell} )$ of $\underline{n}.$

In the case of $H$ to be $m$-relation free with a basis $U^i\subset H^i,$ $i\leq m,$
 we have that
the minimal multiplicative resolution $RH$ of $H$ can be built  by taking $\mathcal{V}$
with $\mathcal{V}^{0,i}\approx \mathcal{U}^i,i\leq m,$ and  $\mathcal{V}^{-n,i},n> 0,$
to be the set
 consisting of monomials (\ref{cups1}) for $1\leq i-n\leq m$  (compare \cite{saneFilt}). The verification
of the acyclicity in the negative resolution degrees of $RH$ restricted to  the range
$RH^{(m)}$ is straightforward (see also Remark \ref{1}). Regarding the map $RH(f),$ we
can choose it on $RH^{(m)}$ as follows.  Let $R_0H(f):R_0H(Y)\rightarrow R_0H(X)$ be
determined by $H(f)$ in an obvious way and then define $RH(f)$ for $a\in
\mathcal{V}^{(m)} $ by
\[
RH(f)(a)=\left\{\!\!\!
\begin{array}{llll}
R_0H(f)(a), &  a\in \mathcal{V}^{0,*},\vspace{2mm}\\
R_0H(f)(a_1)\smile_1\cdots \smile_1 R_0H(f)(a_n), & a=a_1\smile_1\cdots \smile_1
a_{n+1},\\& a\in  \mathcal{V}^{-n,*},a_i\in \mathcal{V}^{0,*},n\geq 1,

\end{array}
\right.
\]
and extend to $RH^{(m)}$ multiplicatively.
 Furthermore, $f_X$
and $f_Y$ are assumed to be preserving  the generators of the form (\ref{cups1}) with
respect to the right most association of $\smile_1$ products in question. Since $h$ annihilates monomials  (\ref{cups1}) and  the existence
of formula (\ref{hirsch})
   in a simplicial cochain complex,
$f_X$ and $f_Y$ are automatically   compatible with the differentials involved. Then
the maps $\alpha$ and $\beta$ in (\ref{diagram}) also preserve $\smile_1$ products, and
 become homotopic by an  $(\alpha,\beta)$-derivation homotopy
 $s: RH(Y)\rightarrow C^*(X)$ defined as follows:
  choose $s$  on $\mathcal{V}^{0,*}$ by $ds=\alpha-\beta$ and
 extend on $\mathcal{V}^{-n,*}$ inductively by
\[s(a_0\smile_1z_n)=-\alpha(a_0)\smile_1 s(z_n)+s(a_0)\smile_1 \beta(z_n)+s(z_n)s(a_0),\ \ \ n\geq 1,\]
in which $z_1=a_1$ and $z_k=a_1\smile_1\cdots \smile_1a_k$ for $k\geq 2,\,a_i\in \mathcal{V}^{0,*}.$  Clearly,
$H(f)=0$ implies $RH(f)|_{V^{(m)}}=0.$ Since (\ref{composition}),
   $D(f)=0$ and so
$f$ is null homotopic by Theorem \ref{Dnull}. Theorem is proved. \vspace{0.2in}

\begin{remark}\label{1} Let $\mathcal{V}^{(m)}_n$ be a subset of  $\mathcal{V}^{(m)}$
 consisting of
all monomials formed by the $\cdot $ and
$\smile_1$ products evaluated on a string of variables $a_1,...,a_{n}.$ Then there is a bijection of
$\mathcal{V}^{(m)}_n$
 with the set of all faces of
  the permutahedron $P_n$
(\cite{milgram},\,\cite{SU2}) such that the resolution differential $d$ is compatible with the cellular differential of $P_n$    (compare \cite{MacLane}).
In particular, the monomial $a_1\smile_1\cdots\smile_1a_n$ is assigned to the top cell of $P_n,$ while the monomials $a_{\sigma(1)}\cdots a_{\sigma(n)},\sigma\in S_n,$
to the vertices of $P_n$ (see Fig. 1 for $n=3$). Thus, the acyclicity of $P_n$ immediately implies the acyclicity of $RH^{(m)}$ in the negative resolution degrees as desired.

\end{remark}

\newpage

\unitlength=1.00mm \special{em:linewidth 0.4pt} \linethickness{0.4pt}
\begin{picture}(87.67,40.33)
\ \ \ \ \ \ \ \ \ \ \ \ \ \ \ \ \ \ \ \ \ \ \ \ \put(41.66,36.67){\line(1,0){30.33}}
\put(41.66,6.67){\line(0,1){30.00}} \

\put(71.99,36.67){\line(0,-1){30.00}}

 \put(71.99,6.67){\line(-1,0){30.33}}
\put(71.99,36.67){\circle*{1.33}}
 \put(41.66,36.67){\circle*{1.33}}
\put(41.66,22.00){\circle*{1.33}}
 \put(41.66,6.67){\circle*{1.33}}
\put(71.99,6.67){\circle*{1.33}}

 \put(71.99,21.67){\circle*{1.33}}

 \put(56.67,22.00){\makebox(0,0)[cc]{$a\smile_1b\smile_1c$}}
\put(57.00,2.67){\makebox(0,0)[cc]{$(a\smile_1b)c$}}
\put(57.67,40.33){\makebox(0,0)[cc]{$c(a\smile_1b)$}}
\put(31.33,13.67){\makebox(0,0)[cc]{$a(b\smile_1c)$}}
\put(83.67,14.00){\makebox(0,0)[cc]{$b(a\smile_1 c)$}}

\put(31.00,28.67){\makebox(0,0)[cc]{$(a\smile_1c)b$}}
\put(83.67,29.33){\makebox(0,0)[cc]{$(b\smile_1c)a$}}
\end{picture}

\begin{center}Figure 1.
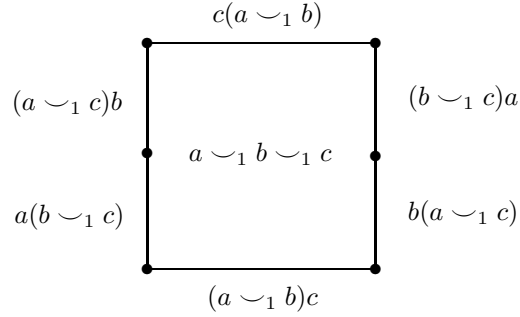
Geometrical interpretation of some syzygies  involving $\smile_1$ product as homotopy
for commutativity  in the resolution $RH.$
\end{center}
\vspace{0.2in}
\begin{remark}
 An example  provided by the Hopf
map $f:S^3\rightarrow S^2$ shows
 that the implication $H(f)=0 \Rightarrow RH(f)|_{V^{(k)}}=0,\,k<m$ for $RH(f)$ making (\ref{diagram}) commutative up to
  $(\alpha,\beta)$-derivation homotopy
 is not true in general. More precisely,
let $x\in R^0H^2(S^2)$ and $y\in R^0H^3(S^3)$ with $\rho x\in H^2(S^2)$ and $\rho y\in
H^3(S^3)$ to be the generators, and let
   $x_1\in R^{-1}H^4(S^2)$ with $dx_1=x^2.$ Then
   $s(x^2)=\alpha(x)s(x)$ is a cocycle in $ C^3(S^3)$  with  $d_{S^3}s(x)=\alpha(x)$ (since $\beta=0$)   and  $[\alpha(x)s(x)]=\rho y .$
  Consequently, while  $H(f)=0=R^0H(f),$
   a map $RH(f):RH(S^2)\rightarrow RH(S^3)$ required in (\ref{diagram})
   has a non-trivial component increasing the resolution degree: Namely,
   $R^{-1}H^4(S^2)\rightarrow R^0H^3(S^3),\, x_1\rightarrow y.$
\end{remark}

\noindent\emph{Proof of Theorem  \ref{null2}.} The conditions that
 $u_i: \pi_i(\Omega Y)\rightarrow H_i(\Omega Y)$ is an inclusion and
$\operatorname{Tor}\left(H^{i+1}(X), H_i(\Omega Y)/\pi_i(\Omega Y)\right)=0$ for $1\leq i<m,$
immediately implies $(1.1)_m.$ So the theorem follows from Theorem 1.

\vspace{0.1in}

\noindent\emph{Proof of Theorem  \ref{null3}.}
 Since the homotopy equivalence $\Omega BG\simeq
G,$
  the conditions of
Theorem 2 are satisfied: Indeed,  there is the following
 commutative diagram
\begin{equation*}
\begin{array}{ccccc}
  \pi_k(G) & \overset{u_k}\longrightarrow & H_k(G) \vspace{1mm}\\
       i_{\pi} \downarrow  & &     \downarrow  i_{H}  \\
\pi_k(G) \otimes \mathbb Q & \overset{u_k\otimes 1}\longrightarrow & H_k(G)
\otimes \mathbb Q
\end{array}
\end{equation*}
where $i_{\pi},i_{H}$ and $u_k\otimes 1$ are the standard inclusions (the last one is a
consequence of a theorem of Milnor-Moore). Consequently,
 $u_k:\pi_k(\Omega BG) \rightarrow  H_k(\Omega BG),k< m,$ is an inclusion, too.
 Theorem is proved.
 \hspace{2.5in} $\Box$

\vspace{5mm}

\end{document}